\documentclass[11pt]{article}

\usepackage{enumerate, amsmath, amsthm, amsfonts, amssymb, xy,  mathrsfs, graphicx, paralist, ulem}
\usepackage[usenames, dvipsnames]{color}
\usepackage[margin=1in]{geometry} 
\usepackage{hyperref}

\input xy
\xyoption{all}

\numberwithin{equation}{section}
\newtheorem{theorem}[equation]{Theorem}

\newtheorem{proposition}[equation]{Proposition}

\theoremstyle{definition}
\newtheorem{rmk}[equation]{Remark}
\newenvironment{remark}[1][]{\begin{rmk}[#1] \pushQED{\qed}}{\popQED \end{rmk}}
\newtheorem{eg}[equation]{Example}

\newtheorem{defn}[equation]{Definition}

\newcommand{\bF}{\mathbf{F}}
\newcommand{\cF}{\mathcal{F}}

\newcommand{\bG}{\mathbf{G}}

\newcommand{\rH}{\mathrm{H}}

\newcommand{\bK}{\mathbf{K}}

\newcommand{\cO}{\mathcal{O}}

\newcommand{\bQ}{\mathbf{Q}}
\newcommand{\cQ}{\mathcal{Q}}

\newcommand{\cR}{\mathcal{R}}

\newcommand{\bZ}{\mathbf{Z}}

\newcommand{\bv}{\mathbf{v}}

\renewcommand{\phi}{\varphi}
\renewcommand{\emptyset}{\varnothing}

\newcommand{\arxiv}[1]{\href{http://arxiv.org/abs/#1}{{\tt arXiv:#1}}}
\newcommand{\bysame}{---------}

\makeatletter
\def\Ddots{\mathinner{\mkern1mu\raise\p@
\vbox{\kern7\p@\hbox{.}}\mkern2mu
\raise4\p@\hbox{.}\mkern2mu\raise7\p@\hbox{.}\mkern1mu}}
\makeatother

\renewcommand{\hom}{\operatorname{Hom}}

\DeclareMathOperator{\rank}{rank}

\DeclareMathOperator{\Pf}{Pf}
\DeclareMathOperator{\Sym}{Sym}

\DeclareMathOperator{\depth}{depth}

\DeclareMathOperator{\Spec}{Spec}

\newcommand{\Gr}{\mathbf{Gr}}

\title{Koszul homology of codimension 3 Gorenstein ideals}
\author{Steven V Sam \and Jerzy Weyman}
\date{March 14, 2012}

\begin{document}

\maketitle

\begin{abstract}
In this note, we calculate the Koszul homology of the codimension 3 Gorenstein ideals. We find filtrations for the Koszul homology in terms of modules with pure free resolutions and completely describe these resolutions. We also consider the Huneke--Ulrich deviation 2 ideals.
\end{abstract}

\section*{Introduction.}

For the codimension 3 Pfaffian ideal of $2n\times 2n$ Pfaffians of a $(2n+1)\times (2n+1)$ generic skew-symmetric matrix, we give an explicit description of the Koszul homology modules. By a result of Buchsbaum--Eisenbud \cite{be}, the general case of codimension 3 Gorenstein ideals reduces to this case. They are filtered by equivariant modules $M_i$ with self-dual pure free resolutions of length 3 supported in the ideal of Pfaffians. The free resolutions of the modules $M_i$ give natural generalizations of the Buchsbaum--Eisenbud complexes for codimension 3 Gorenstein ideals and are interesting in their own right. It was known that the Koszul homology modules of codimension 3 Pfaffian ideals are Cohen--Macaulay \cite[Example 2.2]{huneke}, but no explicit description was given. The only other example we could find in the literature of explicit calculations of Koszul homology is the paper of Avramov--Herzog \cite{ah} which handles the case of codimension 2 perfect ideals. The Koszul homology modules of codimension 3 Pfaffian ideals also gives examples of modules with pure filtrations that do not follow from the results in \cite{numerics}. Finally we calculate the Koszul homology modules for the Huneke--Ulrich deviation 2 ideals which we were studied by Kustin \cite{kustin}.

\subsection*{Acknowledgements.}

Steven Sam was supported by an NDSEG fellowship. Jerzy Weyman was partially supported by NSF grant DMS-0901185. The computer algebra system Macaulay2 \cite{M2} was very helpful for finding the results presented in this paper.


\section{Koszul homology.}

Throughout $R$ is a Cohen--Macaulay (graded) local ring. After this section, we will be working over polynomial rings with $\bZ$-coefficients, which we pretend is a graded local ring by saying that its maximal ideal is the one generated by the variables. Let $I \subset R$ be a (graded) ideal of grade $g$, and let $\mu(I)$ denote the smallest size of a generating set of $I$. The Koszul homology of $I$ depends on a set of generators, but any two choices of {\it minimal} generating sets yield isomorphic Koszul homology. In the case of a minimal generating set, we denote the Koszul homology by $\rH_\bullet(I; R)$. We will only be interested in Koszul homology for minimal generating sets of $I$. We say that $I$ is {\bf strongly Cohen--Macaulay} if the Koszul homology of $I$ is
Cohen--Macaulay.

If $R$ is Gorenstein and $R/I$ is Cohen--Macaulay, then the top nonvanishing Koszul homology $\rH_{\mu(I)-g}(I; R)$ is the canonical module $\omega_{R/I}$ of $R/I$ (see \cite[Remark 1.2]{hunekestrong}). Furthermore, the exterior multiplication on the Koszul complex induces maps
\begin{align} \label{eqn:koszuldual}
\rH_i(I;R) \to \hom_R(\rH_{\mu(I)-g-i}(I;R), \rH_{\mu(I)-g}(I;R)),
\end{align}
and these maps are isomorphisms in the case that $I$ is strongly
Cohen--Macaulay. This is also true if we only assume that the Koszul
homology modules are reflexive \cite[Proposition 2.7]{hunekestrong}.


\section{Codimension 3 Pfaffian ideals.}

In this section we work over the integers $\bZ$ and set $A=\Sym (\bigwedge^2 E)$
where $E$ is a free $\bZ$-module of rank $2n+1$. We consider the ideal
$I= \Pf_{2n}(\phi)$ of $2n\times 2n$ Pfaffians of the generic
skew-symmetric matrix
\[
\phi= (\phi_{i,j})_{1\le i,j\le n}
\]
where $\phi_{i,j}$ are the variables satisfying $\phi_{i,j} = - \phi_{j,i}$. The free resolution for this ideal and its main properties can be found in \cite{be} (the quotient $A/I$ is also the module $M_0$ defined in the next section). Thus if $\lbrace e_1,\ldots, e_{2n+1}\rbrace$ is a basis in $E$,
we can think of $\phi_{i,j}=e_i\wedge e_j\in\bigwedge^2 E$.  Denote
the $2n\times 2n$ Pfaffians of $\phi$ by
\[
Y_i = (-1)^{i+1} \Pf \phi(i)
\]
where $\phi(i)$ is the skew-symmetric matrix we get from $\phi$ by omitting
the $i$-th row and $i$-th column. 

Consider the Koszul complex $\bK_\bullet = \bK(Y_1,\ldots,
Y_{2n+1}; A)$. In this case,
\begin{align*} 
\bK_i = \bigwedge^i (\bigwedge^{2n} E) \otimes A(-in) =
\bigwedge^i E^* \otimes (\det E)^i \otimes A(-in) = \bigwedge^{2n+1-i} E
\otimes (\det E)^{i-1} \otimes A(-in).
\end{align*}


\subsection{Modules $M_i$.}

Before we start we describe a family of $A$-modules supported in the
ideal $I$. For $i=0,\dots,n-1$, we get equivariant inclusions
\begin{align*}
  d_1 &\colon \bigwedge^{2n-i} E \subset \bigwedge^i E \otimes
  \bigwedge^{2n-2i} E \subset \bigwedge^i E \otimes
  \Sym^{n-i}(\bigwedge^2 E)\\
  d_2 &\colon \det E \otimes \bigwedge^{i+1} E \subset
  \bigwedge^{2n-i} E \otimes \bigwedge^{2i+2} E \subset
  \bigwedge^{2n-i} E \otimes \Sym^{i+1}(\bigwedge^2 E) \\
  d_3 &\colon \det E \otimes \bigwedge^{2n+1-i} E \subset \det E
  \otimes \bigwedge^{i+1} E \otimes \bigwedge^{2n-2i} E \subset \det E
  \otimes \bigwedge^{i+1} E \otimes \Sym^{n-i}(\bigwedge^2 E),
\end{align*}
where in each case the first inclusion can be defined in terms of
comultiplication, and the second is given by Pfaffians. We make these
maps more explicit. Let $e_1, \dots, e_{2n+1}$ be an ordered basis for
$E$ compatible with $\phi$. For an ordered sequence $I=(i_1 ,\ldots
,i_n)$ consisting of elements from $[1,2n]$ we denote by $e_I$ the
decomposable tensor $e_{i_1} \wedge \cdots \wedge e_{i_n}$. The embedding
$\bigwedge^{2d} E \subseteq \Sym^d(\bigwedge^2 E)$ sends the tensor
$e_I$ ($\# I=2d$) to the Pfaffian of the $2d\times 2d$ skew-symmetric
submatrix of $\phi$ corresponding to the rows and columns indexed by
$I$. We will denote this Pfaffian by $\Pf(I)$. With these conventions,
the maps $d_1$, $d_2$, $d_3$ are given by the formulas
\begin{align*}
  d_1 (e_I ) &= \sum_{I'\subset I} {\rm sgn}(I', I'')e_{I'}\otimes
  \Pf(I''),\\
  d_2 (e_{[1,2n+1]}\otimes e_J ) &= \sum_{I'\subset [1,2n+1]} {\rm
    sgn}(I', I'') {\rm sgn}(I'', J) e_{I'}\otimes \Pf(I''\cup
  J), \\
  d_3(e_I ) &=\sum_{I'\subset I} {\rm sgn}(I', I'')e_{I'}\otimes
  \Pf(I''),
\end{align*}
where $I''$ is the complement of $I'$ in $I$, all subsets are listed
in increasing order, and ${\rm sgn}(I' ,I'')$ is the sign of the
permutation that reorders $(I', I'')$ in its natural order. The symbol
$\Pf(I''\cup J)$ is by convention 0 if $I''\cap J\ne \emptyset$.

\begin{proposition} \label{prop:pfaffiancomplex}
For $i=0,\dots,n-1$, we define the complex $C^i$
\[
0\rightarrow \begin{array}{c} (\det E)\otimes\bigwedge^{2n+1-i}E \\
  \otimes A(-2n+i-1) \end{array} \xrightarrow{d_3} \begin{array}{c}
  (\det E)\otimes\bigwedge^{i+1}E\\ \otimes A(-n-1) \end{array}
\xrightarrow{d_2} \begin{array}{c} \bigwedge^{2n-i}E \\ \otimes
  A(-n+i) \end{array} \xrightarrow{d_1} \begin{array}{c} \bigwedge^i E
  \\ \otimes A\end{array},
\]
using the inclusions defined above. This complex is acyclic, and the
cokernel $M_i$ is supported in the variety defined by the Pfaffians of
size $2n$.
\end{proposition}

\begin{proof} 
To check that the above is a complex, it is enough to extend scalars to $\bQ$. In this case, we can use representation theory (namely, Pieri's formula \cite[Corollary 2.3.5]{weyman} and the decomposition of $\Sym(\bigwedge^2)$ into Schur functors \cite[Proposition 2.3.8]{weyman}) to see that these maps define a complex. 

To prove acyclicity, we use the Buchsbaum--Eisenbud exactness criterion. The formulation of this result that we use, which is a consequence of \cite[Theorem 20.9]{eisenbud}, is: Given a finite free resolution $\bF_\bullet$ of length $n$, then $\bF_\bullet$ is acyclic if and only if the localization $(\bF_\bullet)_P$ is acyclic for all primes $P$ with $\depth A_P < n$. Localizing at a prime $P$ with depth at most $2$, some variable becomes a unit, so using row and column operations, we can reduce $\phi$ to the matrix
\[
{\hat \phi}=\begin{bmatrix}
0&1&0\\
-1&0&0\\
0&0&\phi'
\end{bmatrix}
\]
where $\phi'$ is a generic $(2n-1)\times (2n-1)$ skew-symmetric matrix. Let $C'^0, \ldots ,C'^{n-2}$ be the complexes in Proposition~\ref{prop:pfaffiancomplex} defined for the matrix $\phi'$. Then
\[
(C^i)_P \cong C'^i \oplus 2C'^{i-1} \oplus C'^{i-2},
\]
with the convention that $C'^{n-1}=0$ and $C'^j=0$ for $j<0$. By induction on the size of $\phi$, we see that each $C'$ is acyclic.
\end{proof}


\subsection{Results and proofs.}

\begin{theorem}
Set $\bK_\bullet = \bK(Y_1,\ldots,Y_{2n+1}; A)_\bullet$.
\begin{compactenum}[\rm (a)]
\item For $0\le j\le n-1$ we have a filtration $\cdots \subset \cF_1 \rH_j \subset \cF_0 \rH_j = \rH_j(\bK_\bullet)$ such that
\[
\cF_i \rH_j / \cF_{i+1} \rH_j \cong M_{j-2i} \otimes (\det E)^j
\]
\item For $0\le j\le n-2$ we have a filtration $0 = \cF_0 \rH_{2n-2-j} \subset \cF_1 \rH_{2n-2-j} \subset \cdots$ such that
\[
\cF_{i+1} \rH_{2n-2-j} / \cF_i \rH_{2n-2-j} \cong M_{j-2i} \otimes (\det E)^{2n-2-j}
\]
\end{compactenum}
\end{theorem}

\begin{proof}
We assume that $n>0$ since the case $n=0$ is trivial. We will construct a sequence of complexes $\bF(r)_\bullet$ for $r=0,\dots,n-1$ such that 
\begin{compactenum}
\item $\bF(0)_\bullet = \bK_\bullet$,
\item $\bF(r)_\bullet$ is concentrated in degrees $[r,2n+1]$,
\item The cokernel of $\bF(r)_\bullet$ has a filtration as specified by the theorem. Letting $\bG(r)_\bullet$ be its minimal free resolution, we have that  $\bF(r+1)_\bullet$ is the minimal subcomplex of the mapping cone $\bF(r)_\bullet \to \bG(r)_\bullet$,
\end{compactenum}
The existence of this sequence implies the first part of the theorem. For the second part, we appeal to \eqref{eqn:koszuldual} which says that $\rH_{2n-2-i}(\bK_\bullet)$ is the $A$-dual of $\rH_i(\bK_\bullet)$ (note that the $M_i$ are self-dual by the form of their free resolutions). We construct this sequence by induction on $r$. 

For $r=0$, there is nothing to check, so assume that $r>0$ and that $\bF(r-1)_\bullet$ has the listed properties. Then $\bF(r-1)$ is the minimal subcomplex of some extension of 
\[
\bK_\bullet \oplus \bigoplus_{i=0}^{r-2} ((\det E)^i \otimes \bigoplus_k C^{i-2k}[-i+1])
\]
which is concentrated in degrees $[r-1,2n+1]$. Since each complex $C$ has length 3, we see that $\bF(r-1)_i = \bK_i$ for all $i \ge r+1$. Recall that $r \le n-1$. Then we see that from the structure of the representations in the resolutions of the $M_i$ that after cancellations, we get (there are no cancellations in homological degree $r+1$)
\begin{align*}
\bF(r-1)_{r-1} &= \bigoplus_k \bigwedge^{r-1-2k} E \otimes (\det E)^{r-1} \otimes A\\
\bF(r-1)_r &= (\bigoplus_k \bigwedge^{2n-r+1+2k} E \otimes (\det E)^{r-1} \otimes A) \oplus \begin{cases} 0 & \text{if $r-1$ is even}\\ (\det E)^r \otimes A  & \text{if $r-1$ is odd} \end{cases}\\
\bF(r-1)_{r+1} &= \bK_{r+1} = \bigwedge^{2n-r} E \otimes (\det E)^r \otimes A
\end{align*}
(we ignore the grading since it is determined by the degree of the functor on $E$). By our induction hypothesis, up to a change of basis, we can write the presentation matrix for $\bF(r-1)$ in ``upper-triangular form'', i.e., the map from $\bigwedge^{2n-r+1+2k'} E$ to $\bigwedge^{r-1-2k} E$ is nonzero if and only if $k' \ge k$. Also, when $r-1$ is odd, the extra term $(\det E)^r \otimes A$ is a redundant relation. Now consider the mapping cone
\[
\xymatrix{
& \bF(r-1)_{r-1} \ar[dl] & \ar[l] \ar[dl] \bF(r-1)_r & \ar[l] \ar[dl] \bF(r-1)_{r+1} & \ar[l] \ar[dl] \cdots \\
\bG(r-1)_0 & \ar[l] \bG(r-1)_1 & \ar[l] \bG(r-1)_2 & \ar[l] \bG(r-1)_3 & \ar[l] 0 
}
\]
The maps $\bF(r-1)_{r-1} \to \bG(r-1)_0$ and $\bF(r-1)_r \to \bG(r-1)_1$ are isomorphisms, except when $r-1$ is odd, in which case the term $(\det E)^r \otimes A$ is in the kernel of the second map. When $r=n-1$, there is an additional cancellation involving the terms $\bigwedge^n E \otimes (\det E)^n \otimes A$ in $\bF(n-2)_n$ and $\bG(n-2)_2$. 

Finally, we can rearrange the resulting presentation matrix into upper-triangular form as follows. Note that all of the maps in the presentation matrix are saturated maps, i.e., their cokernels are free $\bZ$-modules. This can be shown by induction on $r$. Let $N_r$ be the cokernel of the presentation matrix. Consider the submodule of $N_r$ generated by $\bigoplus_{k > 0} \bigwedge^{r-2k} E \otimes (\det E)^r \otimes A$. The quotient is generated by $\bigwedge^r E \otimes (\det E)^r \otimes A$. By induction, the cokernel of $\bG(r-1)_\bullet$ has $M_{r-1}$ as a factor, so this implies that in the diagonal maps, the map $\bigwedge^{2n-r} E \otimes (\det E)^r \otimes A \to \bigwedge^r E \otimes (\det E)^r \otimes A$ is nonzero. Since all of the maps from the relation module to this term are saturated, they all factor through the relations given by $\bigwedge^{2n-r} E \otimes (\det E)^r \otimes A$ (this follows from the uniqueness of such maps up to sign by Pieri's rule \cite[Corollary 2.3.5]{weyman}). Hence $M_r \otimes (\det E)^r$ is a quotient, and continuing in this way, one can show that $N_r$ has the desired filtration. This finishes the induction and the proof.
\end{proof}

\begin{remark} Since the $M_i$ have pure resolutions, the above result
  shows that the Koszul homology of the codimension 3 Pfaffians have a
  pure filtration in the sense of \cite{numerics}.
\end{remark}


\section{Huneke--Ulrich ideals.}

We continue to work over the integers $\bZ$.

In this section, we study the Huneke--Ulrich ideals, which are defined
as follows. Let $\Phi$ be a generic skew-symmetric matrix of size $2n$
and let $\bv$ be a generic column vector of size $2n$. The Huneke--Ulrich
ideal $J$ is generated by the Pfaffian of $\Phi$ along with the
entries of $\Phi \bv$. It is well known that the ideal $J$ is Gorenstein
of codimension $2n-1$ with $2n+1$ minimal generators, i.e., it has
deviation 2. Since $\rH_2$ is the canonical module, the only interesting Koszul homology group to calculate is $\rH_1$.

The notation is as follows. Let $F$ be a free $\bZ$-module of rank $2n$. We work over the polynomial ring 
\[
A=\Sym (\bigwedge^2 F)\otimes \Sym (F^*) = \bZ[x_{i,j}, y_i]_{1\le i < j\le 2n}
\]
where the variables $x_{i,j}$ are the entries of the generic skew-symmetric matrix $\Phi$ and $y_i$ are coordinates of the generic vector $\bv$. Both $A$ and $J$ are naturally bigraded.

The minimal free resolution $\bF_\bullet$ of Huneke--Ulrich ideals was calculated by Kustin \cite{kustin}. When $n=2$, the ideal $J$ is a codimension 3 Gorenstein ideal, so is covered by the previous section via specialization. We will use $V(-d,-e)$ to denote $V \otimes A(-d,-e)$. For $n \ge 4$, the first three terms of the minimal free resolution are given by
\begin{align*}
\bF_1 &= F(-1,-1) \oplus (\det F)(-n,0)\\
\bF_2 &= \bigwedge^2 F(-2,-2) \oplus \bigwedge^{2n-1} F(-n,-1) \oplus A(-1,-2)\\
\bF_3 &= \bigwedge^3 F(-3,-3) \oplus \bigwedge^{2n-2} F(-n,-2) \oplus F(-2,-3) \oplus (\det F)(-n-1,-2)
\end{align*}
When $n=3$, the same is true except that we omit the term $\bigwedge^3 F(-3,-3)$ from $\bF_3$.

Now we consider the Koszul complex $\bK_\bullet$ on the minimal generating set of $J$. Since $J$ has deviation 2, there are only 2 nonzero Koszul homology modules. We already know that $\rH_2$ is the canonical module of $A/J$. More precisely, we have $\rH_2 = (\det F) \otimes A/J(-n-1,-2)$. 
Let us describe the cycle giving $\rH_2$ precisely. Denote the basis of the module $\bK_1=F\otimes A(-1,-1)\oplus (\det F)\otimes (-n,0)$ by $\lbrace e_1,\ldots, e_{2n}, f\rbrace$.
For $1\le i<j\le 2n$ we denote by $X(i,j)$ the $(2n-2)\times (2n-2)$ skew-symmetric matrix obtained from $X$ by removing the $i$-th and $j$-th row and column.
Then the cycle in $\bK_2$ generating $\rH_2(\bK_\bullet )$ is given by
\[
\sum_{i=1}^{2n}y_i e_i\wedge f-\sum_{1\le i<j\le 2n} (-1)^{i+j} \Pf(X(i,j))e_i\wedge e_j.
\]
Equivariantly, we just have the map
\[
(\det F)\otimes A(-n-1,-2)\rightarrow (\det F)\otimes F\otimes A(-n-1,-1)\oplus\bigwedge^2 F\otimes A(-2,2).
\]
It is easy to check that there exists only one (up to a choice of sign) equivariant $\bZ$-flat (saturated) map to each summand and that there is no such equivariant map in lower degrees. It is clear that our map defines a cycle and that the coset of this cycle in homology is annihilated by $J$, since the Koszul homology modules of a complex $\bK(u_1,\ldots , u_r)$ are always annihilated by the ideal $(u_1,\ldots ,u_r)$. So we get an equivariant map
\[
(\det F)\otimes A/J\rightarrow \rH_2(\bK_\bullet).
\]
A standard application of the acyclicity lemma shows that this map is an isomorphism.

\begin{proposition}
The first Koszul homology module has the presentation
\[
\begin{array}{c} F(-2, -3)\oplus\\
 \bigwedge^{2n-2}F(-n, -2)\oplus\\
 (\det F) \otimes F(-n-1, -1)
\end{array}
\to
\begin{array}{c}
A(-1, -2)\oplus\\
\bigwedge^{2n-1}F(-n, -1)
\end{array}
\to \rH_1 \to 0.
\]
\end{proposition}

\begin{proof}
First note that 
\[
\bK_i = \bigwedge^i(F(-1,-1) \oplus (\det F)(-n,0)) = \bigwedge^i F(-i,-i) \oplus (\det F) \otimes \bigwedge^{i-1} F(1-i-n,1-i).
\]
Since the cokernel of both $\bK_\bullet$ and $\bF_\bullet$ agree and $\bF_\bullet$ is acyclic, we get a lifting $\bK_\bullet \to \bF_\bullet$:
\[
\footnotesize \xymatrix@-1.2pc{
& \ar[dl] A & \ar[l] \ar[dl] {\begin{array}{c} F(-1,-1) \\ (\det F)(-n,0) \end{array}} & \ar[l] \ar[dl] {\begin{array}{c} \bigwedge^2 F(-2,-2) \\ (\det F) \otimes F(-1-n,-1) \end{array}} & \ar[l] \ar[dl] {\begin{array}{c} \bigwedge^3 F(-3,-3) \\ (\det F) \otimes \bigwedge^2 F(-2-n,-2) \end{array}} \\
{\begin{array}{c} A \\ \ \end{array}} & 
\ar[l] {\begin{array}{c} F(-1,-1) \\ (\det F)(-n,0) \end{array}} & 
\ar[l] {\begin{array}{c} \bigwedge^2 F(-2,-2) \\ \bigwedge^{2n-1} F(-n,-1) \\ A(-1,-2) \end{array}} & 
\ar[l] {\begin{array}{c} \bigwedge^3 F(-3,-3) \\ \bigwedge^{2n-2} F(-n,-2) \\ F(-2,-3) \\ (\det F)(-n-1,-2) \end{array}} & \ar[l] \cdots
}
\]
Hence a presentation matrix for $\rH_1$ is given by
\[
\begin{array}{c} F(-2, -3)\oplus\\
 \bigwedge^{2n-2}F(-n, -2)\oplus\\
 (\det F) \otimes F(-n-1, -1)\\
(\det F)(-n-1,-2)
\end{array}
\to
\begin{array}{c}
A(-1, -2)\oplus\\
\bigwedge^{2n-1}F(-n, -1)
\end{array}
\to \rH_1 \to 0.
\]
From \cite[Definition 2.3]{kustin}, we conclude that the relations given by $(\det F)(-n-1,-2)$ are redundant, which finishes the proof.
\end{proof}

Inside the affine space $X = \Spec A =\bigwedge^2 F^*\oplus F$ the subvariety defined by $J$ is
\[
Y = \lbrace (\phi, v)\in X \mid \rank \phi\le 2n-2,\ \phi (v)=0\rbrace .
\]

Let us consider the Grassmannian $\Gr(2, F)$ with the tautological sequence
\[
0\to\cR \to F\times \Gr (2, F)\to \cQ\to 0
\]
where $\cR = \{(f,W) \mid f \in W\}$. Consider the incidence variety
\[
Z =  \lbrace (\phi, v, W)\in X\times \Gr(2, F)\mid v\in W \subset \ker(\phi)\rbrace .
\]
Then $\cO_Z = \Sym(\eta)$ where $\eta =\bigwedge^2 \cQ \oplus \cR^*$. The first projection $q\colon Z \to X$ satisfies $q(Z)=Y$.

\begin{theorem} The nonzero homology of $\bK_\bullet$ is
\begin{align*}
\rH_0(\bK_\bullet ) &= \rH^0(\Gr(2,F); \Sym(\eta)) = A/J,\\ 
\rH_1 (\bK_\bullet ) &= \rH^0 (\Gr(2, F); \cR\otimes \Sym (\eta))(-1,-1),\\
\rH_2 (\bK_\bullet ) &= \rH^0 (\Gr(2, F); \bigwedge^2 \cR \otimes \Sym (\eta))(-2,-2)=\det F \otimes A/J(-n-1,-2).
\end{align*}
\end{theorem}

\begin{proof} First we work over $\bQ$. Using the results in \cite[Chapter 5]{weyman}, one can check that the presentation matrix for $\rH^0(\Gr(2,F); \cR \otimes \Sym(\eta))$ contains the same representations as the presentation matrix for $\rH_1(\bK_\bullet)$. By equivariance, such maps are unique up to sign, so we conclude that they agree. From \cite{kustin}, we know that the coordinate ring of $Y$, and hence its canonical module, are torsion-free over $\bZ$. In particular, the descriptions of $\rH_0$ and $\rH_2$ are independent of characteristic. By a Hilbert function argument, one sees that $\rH_1$ is also a torsion-free $\bZ$-module, so our description extends to $\bZ$-coefficients.
\end{proof}


\small \noindent Steven V Sam, 
Massachusetts Institute of Technology, Cambridge, MA, USA \\
{\tt ssam@math.mit.edu}, \url{http://math.mit.edu/~ssam/}

~

\small \noindent Jerzy Weyman, 
Northeastern University, Boston, MA, USA \\
{\tt j.weyman@neu.edu}, \url{http://www.math.neu.edu/~weyman/}

\end{document}